\documentclass[reqno,12pt]{amsart}

\usepackage[margin=2.5cm]{geometry}
\usepackage[T1]{fontenc}
\usepackage[utf8]{inputenc}
\usepackage[english]{babel}
\usepackage{amsmath,amssymb,amsthm}
\usepackage{mathtools} %for dcases
\usepackage{enumitem}
\usepackage{mymacros}
\usepackage{hyperref}
\usepackage{tikz}
\usepackage{comment}
\usepackage{orcidlink}
\usetikzlibrary{matrix}
\usepackage{xcolor}
\usepackage{graphicx}

\newcommand{\D}{\mathbb{D}}

\newcommand{\C}{\mathbb{C}}
\newcommand{\N}{\mathbb{N}}
\newcommand{\T}{\mathbb{T}}

\newcommand{\A}{\mathcal{A}}
\newcommand{\p}{\mathbb{P}}
\newcommand{\E}{\mathbb{E}}
\newcommand{\rkhs}{\mathcal{H}}

\newcommand{\B}{\mathbb{B}}
\renewcommand{\MR}[1]{}

\title[Random Carleson sequences]{Random Carleson Sequences for the Hardy space on the Polydisc and the Unit Ball }

\author[N. Chalmoukis]{Nikolaos Chalmoukis \orcidlink{0000-0001-5210-8206}}
\address{Dipartimento di Matematica e Applicazioni, Universit\`a degli
Studi di Milano Bicocca, Via R. Cozzi 55,  20125, Milano, Italy}
\email{nikolaos.chalmoukis@unimib.it}

\author[A. Dayan]{Alberto Dayan \orcidlink{0000-0002-7346-4354}}
\address{Fachrichtung Mathematik, Universit\"at des Saarlandes, 66123 Saarbr\"ucken, Germany}
\email{dayan@math.uni-sb.de}

\author[G. Lamberti]{Giuseppe Lamberti \orcidlink{0009-0009-0503-0421}}
\address{Univ. Bordeaux, CNRS, Bordeaux INP, IMB, UMR 5251, F-33400 Talence, France} 
\email{giuseppe.lamberti@math.u-bordeaux.fr}

\thanks{N. Chalmoukis is a member of Indam--Gnampa and partially supported by the Hellenic Foundation for Research and Innovation
(H.F.R.I.) under the ``2nd Call for H.F.R.I. Research Projects to support Faculty Members \&
Researchers'' (Project Number: 4662) and the Indam--Gnampa project ``Function theory in several complex and quaternionic variables'' CUP\textunderscore E53C22001930001.}

\thanks{A. Dayan is partially supported by the Emmy Noether Program of the
German Research Foundation (DFG Grant 466012782)}


\begin{document}

\begin{abstract} 
We study the Kolmogorov $0-1$ law for a random  sequence with prescribed radii so that it generates a Carleson measure almost surely, both for the Hardy space on the polydisc and the Hardy space on the unit ball, thus providing improved versions of previous results of the first two authors and of a separate result of Massaneda. In the polydisc, the geometry of such sequences is not well understood, so we proceed by studying the random Gramians generated by random sequences, using tools from the theory of random matrices.  Another result we prove, and that is of its own relevance, is the $0-1$ law for a random sequence to be partitioned into $M$ separated sequences with respect to the pseudo-hyperbolic distance, which is used also to describe the random sequences that are interpolating for the Bloch space on the unit disc almost surely.
\end{abstract}
\maketitle
\section{Introduction}

Let $\rkhs$ be a reproducing kernel Hilbert space of analytic functions on a domain $X$ of $\C^d$, $d\in\N$. A positive regular measure $\mu$ on $X$ is a Carleson measure for $\rkhs$ if $\rkhs$ embeds continuously inside $L^2(X, \mu)$, namely if there exists a constant $C_\mu$ such that
\begin{equation}
    \label{eqn:carlesonrkhs}\tag{CM}
\|f\|_{L^2(X, \mu)}\leq C_\mu~\|f\|_{\rkhs},\qquad \forall f\in\rkhs.
\end{equation}

Carleson measures have been studied in various settings, as they have important applications in harmonic analysis. In this note, we will consider measures generated by sequences, which are intimately connected to the study of interpolating sequences (see for example \cite{AM02, Aleman2019, Chalmoukis2024, Berndtsson1987}). Let $k$ be the reproducing kernel of $\rkhs$, and let $Z=(z_n)_n$ be a sequence in $X$. Define the measure $\mu_Z$ as
\[
\mu_Z:=\sum_{n \in \bN} \|k_{z_n}\|^{-2}~\delta_{z_n}.
\]
Thanks to the reproducing property of the kernels $(k_{z_n})_n$, condition \eqref{eqn:carlesonrkhs} then becomes
\[
\sum_{n\in\N}|\langle f, \hat{k}_n\rangle|^2\leq C_\mu~\|f\|^2_\rkhs, \qquad \forall f \in \cH,
\]
where $\hat{k}_n:=k_{z_n}/\|k_{z_n}\|$ is the normalized kernel at the point $z_n$. The last inequality is equivalent to the boundedness of the  frame operator $T\colon\rkhs\to\rkhs$, defined formally
\[
T(f):=\sum_{n \in \bN}\langle f, \hat{k}_n\rangle ~\hat{k}_n,
\]
which, in turn, is equivalent to the Gram matrix  
\begin{equation}
\label{eqn:Gram}
G:=(\langle \hat{k}_n, \hat{k}_j\rangle)_{n, j\in\N}
\end{equation}
inducing a bounded operator $G\colon\ell^2\to\ell^2$ (see \cite[Chapter 9]{AM02}). We will refer to sequences that generate a Carleson measure for $\rkhs$ as Carleson sequences. In the literature these sequences are sometimes called Carleson-Newman sequences.

Remarkably, apart from such an operator theoretical reformulation, some well known spaces of analytic functions enjoy also a geometric characterization for such measures. For instance, let $$X=\D^d:=\{z=(z^1, \dots, z^d)\in\C^d\,|\, |z^i|<1\}$$ denote the $d$-dimensional polydisc, and let $\rkhs=H^2_d$ be the Hardy space on $\D^d$, that is, the reproducing kernel Hilbert space on $\D^d$ with kernel 
\[
s(z, w):=\prod_{i=1}^d\frac{1}{1-\overline{w^i}z^i},\qquad z, w\in\D^d.
\]
The scalar product is the following
\[
\langle f,g \rangle_{H^2_d} = \sup_{0\leq r<1} \int_{\bT^d} f(r\zeta) \ol{g(r\zeta)} dm(\zeta),
\]
where $dm$ is the normalized Lebesgue measure on $\bT^d$. We will simply write $\langle \, ,\rangle$ instead of $\langle \, ,\rangle_{H^2_d}$ if no confusion arises.
When $d=1$, we use the standard notation $H^2$, rather than $H^2_1$. Carleson showed in \cite{Carleson62} that if $d=1$ a measure $\mu$ satisfies the embedding condition \eqref{eqn:carlesonrkhs} for the Hardy space if and only if it satisfies the one-box condition, namely if there exists $C_\mu >0$ such that
\begin{equation} \tag{OB}
\label{eqn:onebox}
\mu(S_I) \leq C_\mu |I|
\end{equation}
for all arcs $I\subseteq\T$, where $|I|$ is the arc-length measure and $S_I$ is the Carleson square in $\D$ with basis $I$:
\[
S_I:=\{z\in\D \setminus \{ 0\} \,|\, z/|z|\in I, 1-|z|\leq |I|\}.
\]
Moreover if $Z$ is a sequence in the unit disc that is separated with respect the pseudo-hyperbolic distance, then $\mu_Z$ is a Carlson sequence if and only if $Z$ is \emph{interpolating}, that is, if and only if for any sequence of bounded targets $(w_n)_n$ there exists a bounded analytic function $f$ on the unit disc such that $f(z_n)=w_n$, for all $n$  \cite{Carleson62}.

On the other hand, the geometry of Carleson sequences  for $H^2_d$ and their relation to interpolating sequences seems to be much more complicated if $d\geq 2$. In general no necessary and sufficient condition is known for a measure to be Carleson in the polydisc. Regarding this very interesting problem we refer the reader to the work of Chang \cite{Chang1979} and Carleson's counterexample \cite{Carleson74}. Moreover, in the polydisc an interpolating sequence $Z$ is separated with respect the pseudo-hyperbolic distance and it is a Carleson sequence for the Hardy space \cite{Varopoulos72}, but such two conditions fail to be sufficient for $Z$ to be interpolating \cite{Berndtsson1987}.\\

In order to understand better Carleson sequences in the multi-variable setting, we consider \emph{random sequences with prescribed radii} in the polydisc, and we study the probability of such sequences to generate a Carleson measure for $H^2_d$. A random sequence with prescribed radii in $\D^d$ can be defined as follows. Given a sequence of deterministic radii $(r_n)_{n\in\N}$ in $[0, 1)^d$ and a sequence of i.i.d. random variables $(\theta_n)_{n\in\N}$ defined on the same probability space $(\Omega, \A, \p)$ and uniformly distributed in the $d$-dimensional torus $\T^d$, let $\Lambda=(\lambda_n)_{n\in\N}$ be the random sequence defined by
\[
\lambda_n(\omega):=(r_n^1e^{2\pi i\theta^1_n(\omega)},\dots, r_n^de^{2\pi i\theta^d_n(\omega)}),\qquad n\in\N, \omega\in\Omega.
\]

In order to state our results, it is convenient to introduce a certain counting function. To do so we partition $\D^d$ in the dyadic rectangular regions
\[
A_m:=\{z\in\D^d\,|\, 2^{-(m_i+1)}\leq1-|z^i|<2^{-m_i},\,\, i=1,\dots, d\},\qquad m=(m_1, \dots, m_d)\in\N^d,
\]
and denote by $$N_m:=\#\Lambda\cap A_m,\qquad m\in\N^d$$ the (deterministic) number of points of $\Lambda$ in each $A_m$. Define $|m|:=m_1+\dots+m_d$ for all multi-indices in $\N^d$. Observe that being a Carleson sequence is a tail event since it is independent of any finite number of random variables. Therefore, by Kolmogorov's 0-1 Theorem \cite[Theorem 4.5]{Billingsley1995}, it has probability either $0$ or $1$. We would like to find necessary and sufficient conditions on the sequence $(r_n)_n$ such that the corresponding random sequence is a Carleson sequence almost surely. In  \cite{Rudowicz} Rudowicz proved that for $d=1$ a sequence is almost surely Carleson for the Hardy space if 
\begin{equation}\label{Cochran_condition} \sum_{m \in \bN} 2^{-m} N_m^2 < +\infty. \end{equation}

Recently, in \cite[Theorem 1.1]{chalmoukis22}, with the use of the one box characterization of Carleson measures \eqref{eqn:onebox}, the authors improved Rudowicz's sufficient condition, showing that in order for a random sequence to be Carleson almost surely it is sufficient that for some $\varepsilon, C > 0$
\[
N_m \leq C 2^{(1-\varepsilon)m}.
\]

Moreover, in \cite[Theorem 1.3]{Dayan}, a Rudowicz type sufficient condition has been obtained for Carleson sequences for the Hardy space in the polydisc. 
Our first main result is the sharp version of the  $0-1$ law for random Carleson sequences in $H^2_d$, for all $d\ge1$.
\begin{thm}
\label{thm:Carlesonpolydisc}
Let $d$ be a positive integer and $\Lambda$ be a random sequence in $\D^d$. Then,

\begin{equation}
\label{eqn:0-1Carleson}
\bP\left( \Lambda \,\, \text{is a Carleson sequence for } H^2_d\right) = 
    \begin{dcases}
    1 & \text{if} \,\, N_m \leq 
 C 2^{(1-\varepsilon)|m|}\,\, \text{for some } \varepsilon, C > 0, \\
    0 & \text{otherwise}.
\end{dcases}
\end{equation}
\end{thm}
Notice that for $d\ge2$ the one-box condition \eqref{eqn:onebox} is unavailable, hence the techniques used in \cite{chalmoukis22} do not provide insights to the proof of Theorem \ref{thm:Carlesonpolydisc}. We look instead at the \emph{random Gram matrix}  generated by the kernel vectors associated to the random sequence $\Lambda$ as in \eqref{eqn:Gram}, bypassing the difficulty of not having a geometric characterization of Carleson measures for the Hardy space in the polydisc.

Another ingredient for the proof of Theorem \ref{thm:Carlesonpolydisc} is the   $0-1$ law of random  sequences that can be partitioned into finitely many separated sequences with respect to the pseudo-hyperbolic distance in $\D^d$. More specifically, let 
\[
\rho(z, w):=\max_{i=1, \dots, d}\left|\frac{z^i-w^i}{1-\overline{w^i}z^i}\right|,\qquad z, w\in\D^d\,
\]
denote the pseudo-hyperbolic distance in the polydisc. We say that a sequence $Z=(z_n)_n$ in $\D^d$ is separated if 
\begin{equation}
\label{eqn:ws} \tag{S}
\inf_{n\ne j}\rho(z_n, z_j)>0.
\end{equation}

Random separated sequences have been first studied by Cochran 
\cite{Cochran90} who proved that if $d=1$ then \eqref{Cochran_condition} is a necessary and sufficient condition for a random sequence to be separated almost surely. In \cite{Dayan}, the authors extended Cochran's result to the polydisc. Our next theorem extends this results to finite unions of separated sequences. 
\begin{thm}
 \label{Steinhaus_separation}
Let $\Lambda$ be a random sequence in the polydisc. Then for all $M\in \N$,
\begin{equation}
\label{eqn:FU}    
\p(\Lambda\,\text{is the union of $M$ separated sequences})=\begin{cases}
1\quad\text{if}\,\, &\underset{m\in\N^d}{\sum}N_m^{1+M}2^{-M|m|}<\infty,\\
0\quad\text{if}\,\, &\underset{m\in\N^d}{\sum}N_m^{1+M}2^{-M|m|}=\infty.
\end{cases}
\end{equation}
\end{thm}
It turns out that $\rho(z, w)$ coincides with the largest absolute value that an analytic function on $\D^d$ bounded by $1$ that vanishes at $w$ can attain at $z$. The solution of the same extremal problem in $H^2_d$ gives raise to the distance
\[
\rho_s(z, w):=\sqrt{1- \frac{|\langle {s}_z, {s}_w\rangle|^2}{\Vert s_z \Vert^2 \Vert s_w \Vert ^2}}\qquad z, w\in\D^d.
\]
Since $\rho$ and $\rho_s$ are comparable,  one can replace $\rho$ with $\rho_s$ in \eqref{eqn:ws}, and this describes the same sequences. 
Since any Carleson sequence for $H^2_d$ is the finite union of separated sequences with respect $\rho_s$, \cite[Proposition 9.11]{AM02}, the second half of Theorem \ref{thm:Carlesonpolydisc} is therefore deduced from the second half of Theorem \ref{Steinhaus_separation}.

Concerning interpolating sequences for $H^\infty$ in the polydisc, a sufficient condition of geometric flavor has been obtained by Berndtsson et al. \cite{Berndtsson1987}. It states that if a sequence $Z=(z_n)_n$ satisfies the {\it uniform separation} condition, i.e., 
\begin{equation}
    \inf_{n\in \mathbb{N}} \prod_{j \neq n}\rho(z_j,z_n) > 0,
\end{equation}
then it is interpolating. This is in fact an equivalence for $d=1$ by Carleson's Theorem, but not for $d\geq 2$. In \cite[Question 1]{Dayan}, the authors asked  if a random sequence $\Lambda$, which is almost surely weakly separated, must satisfy, almost surely, the uniform separation condition. Although Theorem \ref{thm:Carlesonpolydisc} doesn't answer the above question, it is coherent with a positive answer since both conditions imply the Carleson measure condition. This might reinforce the belief that there is an affirmative answer to the aforementioned question. \\

The last part of this article is devoted to the study of random Carleson sequences for the Hardy space on the unit ball.
Let $\B_d:=\{z\in\C^d\,|\,\sum_{i=1}^d |z^i|^2<1\}$ denote the $d$-dimensional unit ball on $\C^d$. 
The definition of a random sequence $\Lambda=(\lambda_n)_n$ in the unit ball is reminiscent of the analogous construction on the unit disc: given a deterministic sequence of radii $(r_n)_n$ in $(0, 1)$ and a sequence $(\xi_n)_n$ of i.i.d. random variables defined on a probability space $(\Omega, \A, \p)$ and uniformly distributed on the unit sphere $\partial\B_d$, one defines
\[
\lambda_n(\omega):=r_n\xi_n(\omega),\qquad n\in\N, \omega\in\Omega.
\]
For all $m$ in $\N$, let 
\[
N_m:=\#\Lambda\cap\{2^{-(m+1)}\le 1-|z|<2^{-m}\}\subset\B_d.
\]

The question of whether $\Lambda$ generates almost surely a Carleson measure for some significant spaces of analytic functions on the unit ball has been investigated in \cite{Dayan} and \cite{Massaneda96}.
For all $0\le a< d$, denote by $B^a_d$ the reproducing kernel Hilbert space on $\B_d$ having kernel
\[
k^{(a)}_{w}(z):=\frac{1}{(1-\langle z, w\rangle_{\C^d})^{d-a}}, \qquad z, w\in\B_d,
\]

where $\langle z, w\rangle_{\mathbb{C}^d} :=\sum_{i=1}^d z^i \overline{w^i}.$ For $a=0$, $B^a_d$ is the Hardy space on the unit ball, while for $ 0 < a<d$ one obtains a range of Besov-Sobolev spaces, including the Drury-Arveson space ($a=d-1$). For more information about Besov-Sobolev spaces see \cite{Zhu2007} and \cite{Hartz2023} for the Drury Arveson space.

Regarding random Carleson sequences for $B^a_d$, for all $0<a<d$ the same phenomenon observed in \cite[Theorem 1.4]{chalmoukis22} for the unit disc occurs on in the unit ball: $\Lambda$ generates a Carleson measure for $B^a_d$ almost surely if and only if it is a finite measure (see \cite[Theorem 4.3]{Dayan}). We show that this is not the case for $a=0$, as the $0-1$ law for Carleson sequences for the Hardy space on the unit ball resembles the one for the unit disc and the polydisc, refining a theorem of Massaneda \cite[Theorem 3.2]{Massaneda96}.
\begin{thm}
\label{thm:ball}
   Let $\Lambda$ be a random sequence in the unit ball. Then
\[
\p(\Lambda\,\text{is a Carleson sequence for }\,B^0_d)=\begin{cases}
1\,\, &\text{if}\,\, N_m\leq C 2^{d(1-\varepsilon)m}\,\text{for some}\,\, \varepsilon, C>0, \\
0\,\, &\text{otherwise}.
\end{cases}
\] 
\end{thm}

The paper is structured as follows. Section \ref{sec:fu} contains the proof of Theorem \ref{Steinhaus_separation}. An analogous result for the unit ball have been obtained by Massaneda in \cite[Theorem 3.4]{Massaneda96}. The probabilistic tools that we use are contained in Section \ref{sec:prob_tools}, and they differ from the ones used in Massaneda's work.

As a by-product of our study of random sequences that can be written as the union of finitely many separated sequences, we find in Section \ref{sec:Bloch} the $0-1$ law for a random sequence in the unit disc to be interpolating for the Bloch space.

Section \ref{sec:stein:carleson} contains the proof of the first half of Theorem \ref{thm:Carlesonpolydisc}. The main tool used, Theorem \ref{thm:cheroff}, comes from the theory of random matrices, and it allows us to estimate the probability that some diagonal blocks of the random Gramian $G$ are big in norm. Section \ref{sec:Dirichlet} contains some additional remarks on random Carleson sequences for Dirichlet-type kernels in the polydisc. Finally, in Section \ref{sec:ball} we prove Theorem \ref{thm:ball}.\\

\subsection{Notation} If $f$ and $g$ are positive expressions, we will write $f \lesssim g$ if there exists $C>0$ such that $f\leq C g$, where $C$ does not depend on the parameters behind $f$ and $g$, or $\lesssim_a$ if the implicit constant $C$ depends on $a$. We will simply write $f \simeq g$ if $f \lesssim g$ and $g \lesssim f$. Finally when $f$ and $g$ are expressions for which we can consider the limit of their quotient, with $f \sim g$ we mean that $\lim f/g =1$, while $f \sim_a g$ means that $\lim f/g$ is equal to a constant that depends on $a$.

\section{Union of finitely many separated sequences}
\label{sec:fu}
\subsection{A Probabilistic Tool}
\label{sec:prob_tools}
 Let $N$ and $n$ be two positive integers, and consider the problem of placing $n$ points at random into $N$ boxes, where the boxes are chosen independently for each point, and for all points each box has the same probability of being chosen. We are interested in the random variable $\mu_r(n,N)$ that counts the number of boxes in which there are exactly $r$ points. We want to estimate, for $n,N \to \infty$, the number $\bP(\mu_r(n,N)=1)$.
The tools that we are going to use come from \cite{kolchin78}. Define 
\begin{align*}
    \alpha & :=\frac{n}{N} & p_r & := \frac{\alpha^r e^{-\alpha}}{r!}\\
    \sigma_r^2 & := \frac{\alpha}{1-p_r}\left( 1-p_r-\frac{(\alpha-r)^2}{\alpha}p_r \right) & \alpha_r &  := \frac{\alpha-rp_r}{1-p_r}.
\end{align*}

Consider i.i.d. random variables $\eta_1,\dots,\eta_N$ having Poisson distribution with parameter $\alpha$ and let $\zeta_N = \eta_1 + \dots + \eta_N$. Similarly, define $\eta_i^{(r)}$, $i=1,\dots,N$ i.i.d. random variables with distribution
$$
\bP(\eta_i^{(r)}=l)=\bP(\eta_i=l|\eta_i\neq r)
$$
and by $\zeta_N^{(r)}$ the sum of the $\eta_i^{(r)}$.
\begin{lem}[{\cite[Lemma 1, p.60]{kolchin78}}] \label{kolc_lem1}
    $$
    \bP(\mu_r(n,N)=k)=\binom{N}{k}p_r^k(1-p_r)^{N-k}\frac{\bP(\zeta_{N-k}^{(r)}=n-kr)}{\bP(\zeta_N=n)}.
    $$
\end{lem}

\begin{thm}[{\cite[Theorem 1, p.61]{kolchin78}}] \label{kolc_thm1}
    If $m \to \infty$ and $\alpha m \to \infty$, then for fixed $r \geq 2$,
    $$
    \bP(\zeta_m^{(r)}=l)=\frac{1}{\sigma_r\sqrt{2\pi m}}e^{-\frac{(l-m\alpha_r)^2}{2m\sigma_r^2}}(1+o(1)),
    $$
    uniformly with respect to $\frac{l-m\alpha_r}{\sigma_r\sqrt{m}}$ in any finite interval.
\end{thm}

The precise statement of the theorem is the following. Fixed $r\geq 2$ and $M>0$ and consider the domain of parameters 
\[\cD(M,r):=\{(m,\alpha, l) \in \bN\times (0,\infty) \times \bN : \bigg|\frac{l-m \alpha_r}{\sigma_r \sqrt{m}} \bigg| \leq M  \}. \]
Then, for every $\varepsilon > 0 $, there exists $C_0=C(\varepsilon, M ,r) >0 $ such that 
\[ \bigg|  \bP(\zeta_m^{(r)}=l) \sigma_r\sqrt{2\pi m } \, e^{\frac{(l-m\alpha_r)^2}{2m\sigma_r^2}} - 1 \bigg| < \varepsilon,\]
for all $(m,\alpha,l) \in \cD(M,r) $ such that $m,\alpha m > C_0$.\\
As a Corollary, we prove the following variation of \cite[Theorem 3, p.67]{kolchin78}.

\begin{cor} \label{variation_kolc}
    Suppose that $\alpha \to 0$ and $Np_r \to 0$ for $n,N \to \infty$. Then for fixed $r \geq 2$,
    $$
    \lim_{n,N \to \infty}\frac{\bP(\mu_r(n,N)=1)}{N p_r}=1.
    $$
\end{cor}

\begin{proof}
    From Lemma \ref{kolc_lem1} we have that 
    \begin{equation} \label{lem_used}
    \bP(\mu_r(n,N)=1)=Np_r(1-p_r)^{N-1} \frac{\bP(\zeta_{N-1}^{(r)}=n-r)}{\bP(\zeta_N=n)}.
    \end{equation}
    We want to use Theorem \ref{kolc_thm1} with $m=N-1$ and $l=n-r$ to estimate $\bP(\zeta_{N-1}^{(r)}=n-r)$. The only hypothesis we have to check is if $\frac{l-m\alpha_r}{\sigma_r\sqrt{m}}$ remains bounded with our choices, while the other hypothesis easily hold. Notice that since $\alpha \to 0$ we can assume without loss of generality that $ \alpha < 1$. We have
    \begin{equation} \label{hyp_bounded}
    \frac{l-m\alpha_r}{\sigma_r\sqrt{m}} = \frac{n-r-(N-1)\alpha_r}{\sigma_r\sqrt{(N-1)}}.
    \end{equation}
    Since $p_r \sim_r \alpha^r$ we have that $\sigma_r^2 \sim_{r} \alpha$ and so $\sigma_r\sqrt{(N-1)} \sim_{r} \sqrt{n}$.
    Furthermore
    $$
    n-r-(N-1)\alpha_r = n-r-(N-1)\frac{\alpha-rp_r}{1-p_r} = \frac{rNp_r-np_r-r+\alpha}{1-p_r}
    $$
    which remains bounded since $Np_r \to 0$ and $r$ is fixed. So we have obtained that \eqref{hyp_bounded} goes to $0$ for $n,N\to \infty$ and so it remains bounded. We can finally apply Theorem \ref{kolc_thm1}:
    $$
    \bP(\zeta_{N-1}^{(r)}=n-r) = \frac{1}{\sigma_r\sqrt{2\pi (N-1)}}e^{-\frac{(n-r-(N-1)\alpha_r)^2}{2(N-1)\sigma_r^2}}(1+o(1)) \sim_{r} \frac{1}{\sqrt{2\pi n}}.
    $$
    Since $\zeta_N$ follows a Poisson distribution, using  Stirling formula we have
    $$
    \bP(\zeta_N=n) = \frac{n^n}{n!}e^{-n} \sim \frac{1}{\sqrt{2\pi n}}.
    $$
    Finally,  we have that $(1-p_r)^{N-1} = [(1-p_r)^{\frac{1}{p_r}  }]^{p_r (N-1)}  \sim_r 1$. The result now follows from  \eqref{lem_used}.
\end{proof}

\subsection{ Random Separated Sequences in the Polydisc}

We shall now show how Corollary \ref{variation_kolc} can be applied to prove Theorem \ref{Steinhaus_separation}:

\begin{proof}[Proof of Theorem \ref{Steinhaus_separation}]

For all $m$ in $\N^d$, divide the $d$-dimensional torus $\T^d$ into $2^{|m|}$ dyadic rectangles $\{R^m_j\,|\, j_i=1, \dots, 2^{m_i}, i=1,\dots, d\}$ of side-lengths $2^{-m_1}, \dots, 2^{-m_d}$, i.e.
\[
R^m_j:=\{x=(x_1, \dots, x_d)\in[0, 1)^d\,|\, (j_i-1)/2^{m_i}\leq x_i<j_i/2^{m_i}\},\quad j_i=1,\dots, 2^{m_i}.
\]

Let's also label  $\{\lambda_{m, 1}, \dots, \lambda_{m, N_m}\}$ the random points in $A_m$. For any $I\subseteq \{1, \dots, N_m\}$, let $\Omega(m, I)$ be the event that the arguments of the points $\{\lambda_{{m, i}}\,|\, i\in I\}$ are in the same dyadic rectangle $R^n_{j_0}$. Define
\[
\Omega_{m, M}:=\bigcup_{|I|=M+1}\Omega(m, I)
\]
as the event that $M+1$ of the $N_m$ random arguments of points in $A_m$ fall into the same dyadic rectangle. Moreover, the $\theta_i$ are independent and identically distributed, so the probability of each $\Omega(m, I)$ is equal to $2^{-m(|I|-1)}$.\\

Suppose $\sum_{m\in\N^d} N_m^{1+M} 2^{-|m|M} < \infty$ for some $M \geq 1$.
Since for every $M \in \bN$, $$(M+1)! \geq \frac{(M+2)^{M+1}}{e^{M+1}}$$ then
$$
\binom{N_m}{M+1} \leq \frac{N_m^{M+1}}{(M+1)!} \leq \left( \frac{e}{M+2} \right)^{M+1} N_m^{M+1}.
$$
So we finally obtain
\begin{equation}
\label{eqn:borelcantelli}
\bP(\Omega_{m,M}) \leq \binom{N_m}{M+1} 2^{-|m|M} \leq \left( \frac{e}{M+2} \right)^{M+1} N_m^{M+1} 2^{-|m|M}.
\end{equation}

 This, in principle, doesn't exclude the possibility that many random arguments are  very close to each other with high probability, even if they belong to different rectangles. To take care of such eventuality, one can just shift the rectangles by $2^{-m_i-1}$ modulo $1$ in each direction and repeat the above argument. Therefore \eqref{eqn:borelcantelli} controls the probability that $M+1$ arguments of the points in $A_m$ belongs to a rectangle (non necessarily dyadic) of side-lengths $2^{-m_i}$, $i=1, \dots, d$. Moreover, two points in $A_m$ whose arguments are not in the same rectangle of such dimensions are at a uniform mutual pseudo-hyperbolic distance, hence thanks to  Borel-Cantelli's Lemma, $\Lambda$ can be almost surely partitioned into $M$ weakly separated sequences.

Now suppose $\sum_{m\in\N^d} N_m^{M+1} 2^{-|m|M} = \infty$. We will show that $\Lambda$ is not almost surely the union of $M$ weakly separated sequences by showing that for all $l$ in $\N$, almost surely there are infinitely many clusters of $M+1$ points in $\Lambda$ in the same pseudo-hyperbolic ball of radius $2^{-l}$. Fix $l$ in $\N$, and given $m$ in $\N^d$ divide $A_m$ into $2^{dl}$ regions by refining the dyadic partition that defines $A_m$ $l$ times in each direction. Namely,
\[
A_m^j:=\left\{z\in A_m\,\bigg|\, 2^{-m_i+1}+\frac{j-1}{2^{m_i-1+l}}\leq1-|z^i|<2^{-m_i+1}+\frac{j}{2^{m_i-1+l}}, i=1,\dots, d\right\},\, j_i=1, \dots, 2^l.
\]

Since $A_m$ contains $N_m$ points of $\Lambda$, then there exists one $A^j_m$, say $J_m$, that contains at least $L_m\geq N_m/2^{dl}$ points of $\Lambda$. If $m\oplus l:=(m_1+l,\dots, m_d+l)$, then every point in $\D^d$ whose radii is in $J_m$ and arguments are in the same dyadic rectangle $R^{m\oplus l}_k$ are in a ball of pseudo-hyperbolic distance comparable to $2^{-l}$. Therefore, we need to apply Corollary \ref{variation_kolc} to $\mu_{M}(N,n)$, with $N=2^{|m|+dl}$ and $n=L_m$. By eventually removing some radii from the sequence $(r_m)_{m\in\N^d}$, we can assume without loss of generality that $N_m2^{-|m|}\underset{|m|\to\infty}{\to}0$, while $\sum_{m\in\N^d} N_m^{1+M}2^{-M|m|}$ is still divergent (clearly if the associated random sequences is not the union of $M$ weakly separated sequences almost surely, so it won't be the one associated to the whole sequenced $(r_m)_{m\in\N^d}$). In this setting, using the notations of Corollary \ref{variation_kolc}
\[
\alpha=\alpha_n\leq N_m2^{-|m|-dl}\underset{|m|\to\infty}{\to}0,\quad\qquad p_{M}=p_{M, m}\underset{|m|\to\infty}{\sim}\alpha_m^{M}/M!\underset{|m|\to\infty}{\to}{0},
\]
thus 
\[
\p(\mu_{M}(2^{|m|+dl}, L_m)=1)\underset{|m|\to\infty}{\sim_M}L_m^{1+M}2^{-M|m|}\geq N_m^{1+M}2^{-M|m|-dl(1+M)}.
\]
Thanks to Borel-Cantelli Lemma and the divergence of the series $\sum_{n\in\N^d} N_m^{1+M}2^{-M|m|}$, we obtain that almost surely infinitely many of the regions $(A_m)_{m\in\N^d}$ contain a cluster of $M+1$ points in a pseudo-hyperbolic ball of radius comparable to $2^{-l}$. Since this is true for all $l$, we conclude by taking an intersection of countably many events of probability $1$.

\end{proof}
In particular, this gives the 0-1 for $\Lambda$ to be the finite union of weakly separated sequences almost surely, which for the sake of completeness we formulate via the following three equivalent conditions:
\begin{cor}
\label{cor:stein:separation}
Let $\Lambda$ be a random sequence in $\D^d$. Then the following are equivalent:
\begin{description}
\item[(i)] $\Lambda$ can be partitioned almost surely into finitely many weakly separated sequences,
\item[(ii)] There exists an $M$ in $\N$ such that 
\[
\underset{m\in\N^d}{\sum}N_m^{1+M}2^{-M|m|}<\infty,
\]
\item[(iii)] There exists an $\varepsilon>0$  such that $$N_m\lesssim 2^{(1-\varepsilon)|m|},$$
\item[(iv)] There exists some $\beta>1$ such that
\[
\sum_{m\in\N^d}N_m^{\beta}2^{-|m|}<\infty.
\]
\end{description}
\end{cor}

\section{$\gamma$-Carleson Measures in the Unit Disc and Interpolating Sequences for the Bloch Space}
\label{sec:Bloch}
Let $\gamma \in (0,1)$. A sequence of points $Z:=(z_n)_n \subset \bD$  is a \emph{ $\gamma$-Carleson sequence} if the measure $\mu_{Z,\gamma}:=\sum_n(1-|z_n|^2)^\gamma~\delta_{z_n}$ satisfies the one-box condition
$$
\mu_\gamma(S_I)\lesssim_Z|I|^\gamma.
$$

In \cite[Theorem 1.4]{chalmoukis22} the authors found the $0-1$ law for a random sequence $\Lambda$ in the unit disc to satisfy the $\gamma$ - Carleson condition:

\begin{thm} \label{alpha_Carleson}
    Let $\Lambda$ be a random sequence in $\D$ and let be $\gamma \in (0,1)$. Then,
    $$
    \p(\Lambda \text{ is $\gamma$-Carleson})=\begin{dcases}
    1\quad\text{if}\quad&\sum_{m \in \bN} N_m 2^{-\gamma m}<\infty,\\
    0\quad\text{if}\quad&\sum_{m \in \bN} N_n 2^{-\gamma m}=\infty.
\end{dcases}.
    $$
\end{thm}

As an application of Theorem \ref{alpha_Carleson} and Theorem \ref{Steinhaus_separation}, we give the $0-1$ law for random interpolating sequences for the Bloch space on the unit disc.

A holomorphic function $f:\bD \to \bC$ belongs to the Bloch space $\cB$ if 
$$
\|f\|_\cB = |f(0)| + \sup_{z \in \bD}(1-|z|^2)|f'(z)| < \infty.
$$
We are going to consider interpolating sequences for $\cB$ as described in \cite{Boe04}, where $Z=(z_n)_n$ is said to be interpolating for the Bloch space if for every collection of values $(a_n)_n$ such that $$\sup_{n \neq m}\frac{|a_n-a_m|}{\beta(a_n,a_m)} < \infty$$ there exists a function $f \in \cB$ such that $f(z_n)=a_n$, where $\beta$ is the hyperbolic distance in $\bD$.
This choice for the trace space is motivated by the fact that if $f \in \cB$ then  $|f(z)-f(w)| \leq \|f\|_{\cB}\beta(z,w)$. In \cite{Boe04} B{\o}e and Nicolau characterized such interpolating sequences:

\begin{thm}
\label{thm:boenicolau}
    A sequence of points $Z$ in the unit disc is interpolating for the Bloch space if and only if it can be expressed as a union of at most two separated sequences and there exist $0<\gamma<1$ and $C>0$ such that
        \begin{equation}
        \label{eqn:boe_nicolau_gamma}
        \#\{n \in \N : \rho(z,z_n) < r\} \leq \frac{C}{(1-r)^\gamma}.
        \end{equation}
        for all $z\in\D$.
        \end{thm}
As proved in \cite{seip04} and noted also in \cite{Pascuas2007}, condition \eqref{eqn:boe_nicolau_gamma} is equivalent to $Z$ being $\gamma$- Carleson for some $\gamma<1$.\\
The $0-1$ law for random interpolating sequences for the Bloch space reads as follows:
\begin{thm}
    Let $\Lambda$ be a random sequence in $\D$. Then
    $$
    \bP(\Lambda\text{ is interpolating for $\cB$}) = 
    \begin{dcases}
    1 & \text{if} \quad \sum_{m \in \bN} N_m^3 2^{-2m} < \infty, \\
    0 & \text{if} \quad \sum_{n \in \bN} N_m^3 2^{-2m} = \infty.
\end{dcases}
    $$
\end{thm}

\begin{proof}
    Suppose $\sum N_m^3 2^{-2m} < \infty$. By Theorem \ref{Steinhaus_separation} we know that $\Lambda$ is almost surely an union of $2$ separated sequences. Furthermore we know that
    $$
    N_m \lesssim 2^{\frac{2}{3}m}.
    $$
    
 Take $2/3  < \gamma < 1$, then 
    $$
    \sum_{m \in \bN} N_n 2^{-\gamma m} \leq \sum_{m \in \bN} 2^{-\left( \gamma-\frac{2}{3} \right)} < \infty,
    $$
    and so by Theorem \ref{alpha_Carleson} we have that $\Lambda$ is almost surely $\gamma$-Carleson for $\gamma > 2/3$. By Theorem \ref{thm:boenicolau}, we can conclude that $\Lambda$ is almost surely an interpolating sequence for the Bloch space.

    Suppose now $\sum N_m^3 2^{-2m} = \infty$. Then by Theorem \ref{Steinhaus_separation} we know that
    $$
    \bP(\Lambda \text{ is the union of $2$ separated sequences})=0,
    $$
    thus $\Lambda$ is almost surely not interpolating for $\cB$.
\end{proof}

\section{Random Carleson Measures in the Polydisc}

\label{sec:stein:carleson}

\subsection{Preliminaries on Gramians}
Let $H$ be a Hilbert space, and let $V=(v_n)_{n\in\N}$ be a sequence in $H$. The associated \emph{restriction map} $R_V\colon H\to\bC^\bN$ is defined as
\[
R_V(h):=\left(\left\langle h, v_n\right\rangle\right)_{n\in\N},\qquad h\in H.
\]
The sequence $V$ is said to be a \emph{ Bessel system} if $R_V$ maps $H$ continuously into $\ell^2$. In particular, 
\[
V\,\text{is a Bessel system}\iff T_V:=(R_V)^*R_V\,\text{is bounded}\iff G_V:=R_V(R_V)^*\,\text{is bounded}.
\]

The operator $T_V\colon H\mapsto H$ is usually referred to as the \emph{frame operator} associated to $V$, and acts as follows:
\[
T_V(h):=\sum_{n\in\N}\langle h, v_n\rangle~v_n,\qquad h\in H.
\]

The operator $G_V\colon\ell^2\to\ell^2$ is the \emph{ Gramian} of the sequence $V$, and, with respect to the standard basis of $\ell^2$, it is represented by the infinite matrix
\[
\left(\langle v_n, v_j\rangle_H\right)_{n, j\in\N}
\]
Let $H=\rkhs_k$ be a reproducing kernel Hilbert space with kernel $k$ on a set $X$, and let $Z=(z_n)_{n\in\N}$ be a sequence in $X$. The measure $$\mu_Z:=\sum_{n\in\N}\|k_{z_n}\|^{-2}~\delta_{z_n}$$
is a Carleson measure if $\rkhs_k$ embeds continuously in $L^2(X, \mu_\Lambda)$. Therefore, $\Lambda$ generates a Carleson measure for $\rkhs_k$ if and only if the sequence of normalized kernels $\left(k_{z_n}/\|k_{z_n}\|\right)_{n\in\N}$ forms a Bessel system in $\rkhs_k$, that is, if and only if the Gram matrix
\[
G_\Lambda:=\left(\left\langle\frac{k_{z_n}}{\|k_{z_n}\|}, \frac{k_{z_j}}{\|k_{z_j}\|}\right\rangle\right)_{n, j\in\N}
\]
defines a bounded operator from $\ell^2$ to itself.\\

The following two Lemmas will be relevant for the proof of Theorem \ref{thm:Carlesonpolydisc}. 

\begin{lem}
\label{lem:overlappingblocks}
Let $\N=\bigcup_{j\in\N} I_j$, where each $I_j$ is finite, and suppose that a sequence $V=(v_n)_{n\in\N}$ in a Hilbert space is such that
\[
\langle v_n, v_k\rangle=0
\]
whenever $n\in I_j$, $k\in I_l$, and $I_j\cap I_l=\emptyset$. Suppose that
\begin{equation}
\label{eqn:finiteoverlaps}
M:=\sup_{j\in\N} \#\{k\,|\, I_k\cap I_j\}<\infty.
\end{equation}
Then
\begin{equation}
    \label{eqn:overlappingblocks}
    \|G_V\|\leq M~\sup_{j}\|G_{V_j}\|,
\end{equation}
where $V_j=(v_n)_{n \in I_j}$.
\end{lem} 
\begin{proof}
Since
\[
\left(\sum_{j}G_{V_j}\right) - G_V
\]
is positive semi-definite, one has that
\[
\|G_V\|\leq\left\|\sum_{j}G_{V_j}\right\|.
\]

Thanks to Cotlar-Stein Lemma, 
\[
\left\|\sum_{j}G_{V_j}\right\|\leq\sqrt{C_0~C_\infty}, 
\]
where
\[
C_0:=\sup_{j}\sum_{l}\sqrt{\|G_{V_j}G_{V_l}^*}\|,\quad C_\infty:=\sup_{l}\sum_{j}\sqrt{\|G_{V_j}^*G_{V_l}\|}.
\]

Thanks to \eqref{eqn:finiteoverlaps}, fixed any $j$, $G_{V_l}$ and $G_{V_j}$ have not orthogonal ranges at most $M$ times, thus \eqref{eqn:overlappingblocks} holds.

\end{proof}
 In what follows, if $h\in \cH$ is a vector of a Hilbert space we will denote by $  h h^* $ the rank one operator on $\cH $ which acts naturally as follows, $h h^* (x) = \langle x,h \rangle  h, \,\, x\in \cH$. Let $Z=(z_n)_{n\in\N}$ be a  sequence in $\D^d$, and let, for all $n$, $S_n$ be the normalized Szegö kernel at $z_n$. The measure $\mu_Z:=\sum_nS_{z_n}(z_n)^{-1}\delta_{z_n}$ is a Carleson measure for $H^2_d$ if and only if the frame operator $T$ associated to the sequence $(S_n)_{n\in\N}\subset H^2_d$ is bounded, where
\[
T(f):=\sum_{n\in\N}\langle f, S_n\rangle~S_n=\left(\sum_{n\in\N}S_nS_n^*\right)(f),\qquad f\in H^2.
\]
Define, for all $a\leq b$ in $\N$,
\[
T_{[a, b]}:=\sum_{|m|=a}^b\sum_{z_j\in A_m}S_jS_j^*,
\]

as the frame operator of the  kernel functions associated to points in the annuli $\{A_m\,|\, a\leq |m|\leq b\}$.  By writing each $S_j$ in their coordinates with respect to the monomials basis  of $H^2_d$
\[
S_j(z)=\left(\prod_{i=1}^d\sqrt{1-|z^i_j|^2}\right)\sum_{l\in\N^d}\overline{z_j}^lz^l\qquad z\in\D^d,
\]
we will be able to use some results on the highest eigenvalue of random matrices that can be written as the sum of rank-one independent components. Since such components must be \emph{finite} dimensional square matrices (see Theorem \ref{thm:cheroff} below), we need first to approximate the vectors $(S_j)_{j\in\N}$ using partial sums. Fix a $L$ in $\N$, and let 
\[
T_{[a, b]}^L:=\sum_{|m|=a}^b\sum_{z_j\in A_m}P_L(S_j)(P_L(S_j))^*
\]
be the frame operator of the collection $\left\{P_L(S_j)\,\bigg|\, z_j\in\overset{b}{\underset{|m|=a}{\bigcup}} A_m\right\}$, where 
\[
P_L(S_j)(z)=\left(\prod_{i=1}^d\sqrt{1-|z^i_j|^2}\right)\sum_{l_1, \dots, l_d\leq L}\overline{z_j}^lz^l\qquad z\in\D^d.
\]
Fixed $a$ and $b$ and set $N_{[a, b]}:=\sum_{|m|=a}^b N_m$. We seek how large must $L$ be for $\|T_{[a, b]}\|$ and $\|T_{[a, b]}^L\|$ to be close.

\begin{lem}
\label{lem:partialsums}
For all $a<b$ in $\N$ and for any $L$ in $\N$
\begin{equation}
\label{eqn:gramapprox}
\|T_{[a, b]}\|\leq \|T^L_{[a, b]}\|+ C N_{[a, b]}(1-2^{-b})^{2L}.
\end{equation}
\end{lem}
\begin{proof}
Thanks to orthogonality,
\[
T_{[a, b]}=T_{[a, b]}^L+\tilde{T}_{[a, b]}^L,
\]
where $\tilde{T}^L_{[a, b]}$ is the frame operator of the collection $\left\{(Id-P_L)(S_j)\,|\, z_j\in\bigcup_{|m|=a}^b A_m\right\}$. Thus $\|T_{[a, b]}\|$ and $\|T_{[a, b]}^L\|$ differs at most by $\|\tilde{T}^L_{[a, b]}\|$. Notice that for all $f$ in $H^2$
\begin{align*}
 & \left\|\sum_{|m|=a}^b \sum_{z_j\in A_m} \langle{f, (Id-P_L)(S_j)}\rangle~(Id-P_L)(S_j)\right\| \\
& \leq N_{[a, b]}~\|f\|~\sup\left\{\|(Id-P_L)(S_j)\|^2\,\bigg|\, z_j\in\bigcup_{|m|=a}^b A_m \right\}\\
& = N_{[a, b]}~\|f\|~\sup\left\{1-\prod_{i=1}^d(1-|z^i_j|^{2L})\,\bigg|\, z_j\in\bigcup_{|m|=a}^b A_m \right\}\\
& \leq N_{[a, b]}~\|f\|~\left(
1-\left(1-\left(1-2^{-b}\right)^{2L}\right)^d\right) \\
& \lesssim_d N_{[a, b]}~\|f\|~(1-2^{-b})^{2L}.
\end{align*}

since for all $z_j\in A_m$, $|m|=a, \dots, b$, one has that 
\[
|z_j^i|\leq1-2^{-m_i}\leq1-2^{-|m|}\leq 1-2^{-b},\qquad i=1, \dots d.
\]

\end{proof}

\subsection{Random Szegö Gramians}

We are now ready for the proof of Theorem \ref{thm:Carlesonpolydisc}. Thanks to Corollary \ref{cor:stein:separation}, if $N_m \not\lesssim 2^{(1-\varepsilon)|m|}$ for every $\varepsilon >0$, then $\Lambda$ is almost surely not the union of finitely many separated sequences. Hence, \cite[Proposition 9.11]{AM02}, $\Lambda$ is not a Carleson sequence almost surely.\\
Let $\Lambda$ be a random sequence in $\D^d$, and assume that $N_m\lesssim 2^{(1-\varepsilon)|m|}$, for some $\varepsilon>0$. The idea is to arrange a (deterministic) decomposition of the associated random Gramian $G_\Lambda$, by writing it as the sum of a sequence of overlapping blocks and the remaining off-blocks part. More precisely, set 
\[
I_j:=[2^{j-1}, 2^j/\varepsilon]\cap\N,\qquad j\in\N.
\]
Let $G_\Lambda=G_1+G_2$, where the entries of $G_1$ coincide with the ones of $G$ on the overlapping blocks $(X_j)_{j\in\N}$, $X_j$ being the Gram matrix of the collection
\[
\{S_n\,|\, \lambda_n\in A_m, |m|\in I_j\},
\]
and they are zero elsewhere. $G_2$ instead is zero on the overlapping blocks and carries the entries of $G_\Lambda$ outside of such blocks. We prove Theorem \ref{thm:Carlesonpolydisc} by showing in two separate steps that $G_1$ and $G_2$ are bounded almost surely:\\

\paragraph{\bf{Diagonal overlapping blocks estimates}}
The key result that we are going to use from the theory of random matrices can be extracted from the matrix Chernoff's inequality, \cite[Theorem 1.1]{Tropp2012}:
\begin{thm}
\label{thm:cheroff}
Let $T$ be the frame operator of finitely many random independent vectors $v_1, \dots, v_N$ in $\C^L$, and let $\mu:=\|\E(T)\|$. If $\|v_j\|\le1$ for all $j$ almost surely, then
\[
\bP\left(\left\|T\right\|\geq (1+\delta)\mu\right)\leq L~\left(\frac{e}{1+\delta}\right)^{\delta\mu}\qquad\delta\geq0.
\]
\end{thm}

A computation shows that by our choice of the sequence $(I_j)_{j}$, the number of overlaps of the blocks that compose $G_1$ is uniformly bounded by $M_\varepsilon:=\log_2(1/\varepsilon)+1$. Therefore, thanks to  Lemma \ref{lem:overlappingblocks} in order to show that $G_1$ is bounded almost surely it suffices to show that
\begin{equation}
\label{eqn:supX_j}
\sup_j\|X_j\|<\infty
\end{equation}
almost surely. Notice that $X_j$ is the Gram matrix associated to the frame operator $T_{[a_j, b_j]}$, where $a_j=2^{j-1}$ and $b_j=2^j/\varepsilon$. Consider $T_{[a_j, b_j]}^{L_j}$, where $L_j=2^{3b_j}$ so that

\begin{equation}
\label{eqn:choiceofL_m}
N_{[a_j, b_j]}(1-2^{-b_j})^{2L_j}\lesssim_\varepsilon 2^{2(1-\varepsilon)b_j}(1-2^{-b_j})^{2^{3b_j}}\underset{b_j\to\infty}{\to}0.
 \end{equation}
 
Thanks to Lemma \ref{lem:partialsums},  $\|X_j\|$ and $\|T^{L_j}_{[a_j, b_j]}\|$ are closer than an uniform constant. We have, with respect to the coordinates given by the monomials in $H^2_d$, 
\[
\begin{split}
T^{L_j}_{[a_j, b_j]}&=\sum_{|m|=a_j}^{b_j}\sum_{\lambda_n\in A_m}P_{L_j}(S_n)(P_{L_j}(S_n))^*\\
&=\sum_{|m|=a_j}^{b_j}\sum_{\lambda_n\in A_m}\left(\prod_{i=1}^d(1-(r^i_n)^2\right)\left(r_n^{k+l}~e^{-i\theta_n(k-l)}\right)_{|l|,|k|=0}^{L_j},
\end{split}
\]
where
\[
r_n^{k+l}~e^{-i\theta_n(k-l)}=\prod_{i=1}^d(r^i_n)^{k_i+l_i}e^{-i\theta^i_n(k_i-l_i)}.
\]

Hence the expectation of $T^{L_j}_{[a_j, b_j]}$ is diagonal, and its norm is 
\[
\mu_j:=\left\|\E\left(T^{L_j}_{[a_j, b_j]}\right)\right\|\simeq \sum_{|m|=a_j}^{b_j}N_m2^{-|m|}\lesssim\sum_{|m|=a_j}^{b_j}2^{-\varepsilon |m|}\lesssim 2^{-\frac{\varepsilon a_j}{2}}.
\]

  Fix a positive number $A$, to be determined later. By applying Theorem \ref{thm:cheroff}, $\delta_j:= \frac{A}{\mu_j}-1$, we obtain
\[
\begin{split}
\bP\left(\left\|T_{[a_j, b_j]}^{L_j}\right\|\geq A\right)&\leq L_j^d\left(\frac{\mu_j~e}{A}\right)^{A-\mu_j}\\
&\underset{j\to\infty}{\sim} L_j^d\left(\frac{\mu_j~e}{A}\right)^{A}\\
&\lesssim_A 2^{3db_j- \frac{A\varepsilon a_j}{2}}\\
&= 2^{a_j\left(\frac{6d}{\varepsilon}-\frac{A\varepsilon}{2}\right)},
\end{split}
\]
since $b_j=2a_j/\varepsilon$. It suffices then to pick $A>\frac{12d}{\varepsilon^2}$, and Borel-Cantelli Lemma gives \eqref{eqn:supX_j}.\\

\paragraph{\bf{Off diagonal estimates}}
We are left with showing that $G_2$ is bounded almost surely, under the assumption that $N_m\lesssim 2^{(1-\varepsilon)|m|}$. The advantage of taking overlapping blocks in the first step of our proof is that the entries that are left composing $G_2$ are far away from the diagonal, hence we can exploit the decay of their expectation. We show that $\E(\|G_2\|_{HS})$, the expectation of the Hilbert-Schmidt norm of $G_2$, is finite, concluding the proof of Theorem \ref{thm:Carlesonpolydisc}. Thanks to \cite[Remark 3.2]{Dayan}, if $\lambda_{n}$ is in $A_{m}$ and $\lambda_{l}$ is in $A_{k}$, then
\begin{equation}
\label{eqn:expentry}
\E(|\left\langle S_{\lambda_{n}}, S_{\lambda_{l}}\right\rangle|^2)=\frac{(1-r^2_n)(1-r^2_j)}{1-r_nr_j}\simeq\prod_{i=1}^d\frac{1}{2^{m_i}+2^{k_i}}.
\end{equation}
Thus the expectation of the square of the Hilbert-Schmidt norm of the off-diagonal block 
\[
\left(\langle S_n, S_l\rangle\right)_{\lambda_n\in A_m, \lambda_l\in A_k}
\]
of $G_\Lambda$ is controlled by $\frac{N_mN_l}{\prod_i2^{m_i}+2^{l_i}}$. Hence 
\[
\begin{split}
\E(\|G_2\|_{HS}^2)\simeq&\sum_{j=1}^\infty\sum_{|m|=2^{j-1}}^{2^j}\sum_{|l|\ge2^j/\varepsilon}N_mN_l\prod_{i=1}^d\frac{1}{2^{m_i}+2^{l_i}}\\
\lesssim &\sum_{j=1}^\infty\sum_{|m|=2^{j-1}}^{2^j}\sum_{|l|\ge2^j/\varepsilon}\frac{2^{(1-\varepsilon)(|m|+|l|)}}{2^{|m|}+2^{|l|}}\\
\leq&\sum_{j=1}^\infty\sum_{|m|=2^{j-1}}^{2^j}2^{(1-\varepsilon)|m|}\sum_{|l|\ge2^j/\varepsilon}2^{-\varepsilon |l|}\\
\lesssim&\sum_{j=1}^\infty\sum_{|m|=2^{j-1}}^{2^j}2^{(1-\varepsilon)|m|}\sum_{s\ge2^j/\varepsilon}s^{d-1}2^{-\varepsilon s}\\
\lesssim&\sum_{j=1}^\infty\sum_{|m|=2^{j-1}}^{2^j}2^{(1-\varepsilon)|m|-2^{j}(1-\varepsilon/2)}\\
\lesssim&\sum_{j=1}^\infty\sum_{r=2^{j-1}}^{2^j}r^{d-1}2^{(1-\varepsilon)r-2^{j}(1-\varepsilon/2)}\\
\leq&\sum_{j=1}^\infty2^{dj-\varepsilon2^{j-1}}<\infty.
\end{split}
\]

This concludes the proof of Theorem \ref{thm:Carlesonpolydisc}.
\subsection{Random Dirichlet Gramians}
\label{sec:Dirichlet}
Given a random sequence $\Lambda$ in the polydisc, one can ask whether it generates a Carleson measure for a reproducing kernel Hilbert space other than the Hardy space. As in the one variable setting, for all $0\le a\le1$, let $D^a_d$ be the associated Dirichlet-type space on $\D^d$, that is, the reproducing kernel Hilbert space having kernel
\[
k^{(a)}_w(z):=\begin{cases}
\prod_{i=1}^d\frac{1}{(1-\overline{w^i}z^i)^{1-a}}\qquad&a\le0<1\\
\prod_{i=1}^d\frac{1}{z^i\overline{w^i}}\log\frac{1}{1-\overline{w^i}z^i}\qquad &a=1
\end{cases}
\]
Let $S^{(a)}_w(z):=k^{(a)}_w(z)/\left\|k^{(a)}_w\right\|$ denote the associated normalized kernel.
The case $a=0$ corresponds to the Hardy space, while the case $a=1$ corresponds to the Dirichlet space on the polydisc. 
In order to study random sequence on the polydisc that are Carleson for the spaces $D^2_d$, we first extend \eqref{eqn:expentry} to this setting:
\begin{lem}
Let $\Lambda=(\lambda_n)_n$ be a random sequence in the polydisc. Then      
\label{lemma:expentry}
\[
\E(|S^{(a)}(\lambda_n, \lambda_j)|^{2})\simeq\prod_{i=1}^d (1-r^i_n)^{1-a}(1-r^i_j)^{1-a}\begin{cases}
\frac{1}{(1-r^i_nr^i_j)^{1-2a}}\quad&0\le a<1/2\\
\log\frac{1}{1-r^i_nr^i_j}\quad&a=1/2\\
1\quad\quad&1/2<a\le1
\end{cases}.
\]
\end{lem}
\begin{proof}
First, observe that it is enough to prove the Lemma for the case $d=1$, since the $d$ coordinates of each random variable $\theta_n$ are independent, and the expectation of the product of independent random variables factorizes. Write then 
\[
k^{(a)}_w(z)=\sum_{l=0}^\infty c_l (z\overline{w})^l\qquad z, w\in\D
\]
where $c_l\simeq (1+l)^{-a}$. Hence
\[
\begin{split}
&\E(|S^{(a)}(\lambda_n, \lambda_j)|^{2})\\
\simeq&(1-r_n)^{1-a}(1-r_j)^{1-a}\sum_{l, r=0}^\infty c_lc_r (r_nr_j)^{l+r}\E(e^{i(l-r)(\theta_n-\theta_j)})\\
=&(1-r_n)^{1-a}(1-r_j)^{1-a}\sum_{l=0}^\infty c_l^2(r_nr_j)^{2l},
\end{split}
\]
and the Lemma follows.
\end{proof}
Since all the kernels involved are invariant under rotations, $\Lambda$ generates for all $\omega$ in $\Omega$ a finite measure for $D^a_d$ if and only if
\begin{equation}
\label{eqn:CarlesonDa}
\sum_{n \in \bN}\|k^{(a)}_{r_n}\|^{-2}=\sum_{m\in\N^d}N_m2^{-(1-a)|m|}<\infty.
\end{equation}
Recall that for $d=1$ and $0<a\le1$, \eqref{eqn:CarlesonDa} is also sufficient for $\Lambda$ to generate a Carleson measure for $D^a$, see \cite[Theorem 1.4]{chalmoukis22}. Thanks to Lemma \ref{lemma:expentry}, this can be seen to be true also in the multi-variable case, though our argument covers only the case $1/2<a\le 1$. Indeed, if \eqref{eqn:CarlesonDa} holds,thanks to Lemma \ref{lemma:expentry}, $1/2<a\le1$, one obtains that
\[
\begin{split}
&\E\left(\sum_{n\ne j}|S^{(a)}(\lambda_n, \lambda_j)|^2\right)=\sum_{n\ne j}\E\left(|S^{(a)}(\lambda_n, \lambda_j)|^2\right)\\
\simeq&\sum_{n\ne j}\|k^{(a)}_{r_n}\|^{-2}\|k^{(a)}_{r_j}\|^{-2}<\infty,
\end{split}
\]
hence the Gram matrix of $\Lambda$ in $D^a_d$
\[
G_\Lambda^a:=\left(S^{(a)}(\lambda_n, \lambda_j)\right)_{n, j}
\]
is almost surely a Hilbert-Schmidt perturbation of the identity. In particular, 
\begin{cor}
\label{coro:CarlesonDa}
Let $d\ge 1$, and let $\Lambda$ be a random sequence in $\D^d$. For all $1/2<a\le1$, 

\begin{equation}
\label{eqn:carleson_dirichlet_poly}
\p(\Lambda\,\text{is a Carleson sequence for }\,D^a_d)=\begin{cases}
1\qquad&\text{if}\,\, \sum_{m\in\N^d}N_m2^{-(1-a)|m|}<\infty\\
0\qquad&\text{if}\,\,\sum_{m\in\N^d}N_m2^{-(1-a)|m|}=\infty
\end{cases}
\end{equation}

\end{cor}
We conjecture that Corollary \ref{coro:CarlesonDa} holds also for $0<a\le1/2$. The random matrix argument used in Theorem \ref{thm:Carlesonpolydisc} for the case $a=0$ can't be used to prove \eqref{eqn:carleson_dirichlet_poly} for $0<a\le1/2$, since it requires that the sequence $(N_m2^{-|m|})_{m\in\N^d}$ decay exponentially. A geometric sufficient condition for Carleson measures for $D^a_d$ in the deterministic setting is only available for $d=2, 3$: see \cite{Mozolyako2022} for the case $0<a<1$ and \cite{Arcozzi2023} for the case $a=1$.

\section{Random Carleson Measures on the Unit Ball}
\label{sec:ball}
The proof of Theorem \ref{thm:ball} relies on some properties of positive semi-definite matrices. An infinite matrix $A=(a_{nj})_{n, j}$, is positive semi-definite, $A\ge 0$, if for any $N > 0$ and $c_1, \dots, c_N$ in $\C$
\[
\sum_{n, j=1}^Nc_n\overline{c_j} a_{nj}\ge 0.
\]

For instance, any Gram matrix associated to a sequence of vectors $(v_n)_n$ in a Hilbert space is positive semi-definite. A noteworthy result about positive semi-definite matrices is that if $A=(a_{nj})_{n, j}$ and $B=(b_{nj})_{n, j}$ are positive semi-definite, then  $A\odot B:=(a_{nj}b_{nj})_{n, j}$ is positive semi-definite as well. As a corollary, one proves the following:
\begin{lem}
\label{lemma:gram}
Let $A\colon\ell^2\to\ell^2$ be a bounded infinite matrix, and let $H$ be a positive semi-definite infinite matrix having all the entries on its main diagonal equal to $1$. Then $\|A\odot H\|\leq\|A\|$.
\end{lem}
\begin{proof}
The norm of $A$ is the least $C>0$ such that
\begin{equation}
    \label{eqn:normpsd}
C^2 Id-A\ge 0.
\end{equation}
By Schur multiplying the left hand side of \eqref{eqn:normpsd} by $H$, we obtain that $C^2Id-A\odot H\ge 0$, hence $\|A\odot H\|\leq C$.
\end{proof}
This, together with \cite[Theorem 4.3]{Dayan} and \cite[Theorem 3.4]{Massaneda96}, provides the proof of Theorem \ref{thm:ball}:

\begin{proof}[Proof of Theorem \ref{thm:ball}]
If $N_m \not \lesssim 2^{d(1-\varepsilon)m}$ for every $\varepsilon >0$, then thanks to \cite[Theorem 3.4]{Massaneda96} $\Lambda$ is not the union of finitely many separated sequences almost surely with respect to the pseudo-hyperbolic metric
\[
\rho(z, w)^2:=1-\frac{(1-|z|^2)(1-|w|^2)}{|1-\langle z, w\rangle_{\C^d}|^2}
\]
hence the Gram matrix
\[
G_\Lambda:=\left(\frac{(1-|z|^2)^\frac{d}{2}(1-|w|^2)^\frac{d}{2}}{\left(1-\langle z, w\rangle_{\C^d}\right)^d}\right)_{n, j}
\]
is not bounded almost surely, and $\Lambda$ does not generate a Carleson measure for the Hardy space. On the other hand, if there exists a positive $\varepsilon$ such that $N_m\lesssim 2^{d(1-\varepsilon)m}$, then $\Lambda$ generates a finite measure for $B^\nu_d$ for some $0<\nu$. Hence, \cite[Theorem 4.3]{Dayan}, $\Lambda$ is a Carleson sequence for $B^\nu_d$ almost surely, and the Gram matrix
\[
G_\Lambda:=\left(\frac{(1-|z|^2)^\frac{d-\nu}{2}(1-|w|^2)^\frac{d-\nu}{2}}{\left(1-\langle z, w\rangle_{\C^d}\right)^{d-\nu}}\right)_{n, j}
\]
is bounded almost surely. But since $G_\Lambda$ is the Schur product between $G_\Lambda^\nu$ and $G_\Lambda^{d-\nu}$ then $G_\Lambda$ is bounded almost surely thanks to Lemma \ref{lemma:gram}, hence $\Lambda$ is Carleson for the Hardy space almost surely.

\end{proof}

\section*{Acknowledgements} The authors would like to express their gratitude to Roland Speicher and Marwa Banna for directing us towards the literature on random matrices that contain Theorem \ref{thm:cheroff}. We would also like to thank the anonymous referee for the  careful reading of the manuscript.

\bibliographystyle{abbrv}
\bibliography{literature}

\end{document}